\documentclass[12pt]{article}
\usepackage{amsmath,amssymb,amsthm,amscd,latexsym,}
\usepackage{graphicx}

\makeatletter
\renewcommand{\mod}[1]{\allowbreak \if@display \mkern 8mu \else
\mkern 5mu\fi {\operator@font mod}\,\,#1}
\makeatother

\newcommand{\bc}{\mathbb C}

 \newcommand{\bq}{\mathbb Q}
\newcommand{\br}{\mathbb R}
\newcommand{\bz}{\mathbb Z}

\DeclareMathOperator{\rk}{rk}

\newtheorem{theorem}{Theorem}

\newtheorem{definition}{Definition}

\numberwithin{proposition}{section}
\numberwithin{definition}{section}
\numberwithin{corollary}{section}
\numberwithin{remark}{section}
\numberwithin{lemma}{section}
\numberwithin{equation}{section}
\numberwithin{theorem}{section}
\numberwithin{conjecture}{section}
\numberwithin{example}{section}

\newcommand\La{\mathcal L}

\newcommand\gggg{\mathfrak g}

\newcommand{\cD}{\mathcal D}

\newcommand{\latt}[1]{{\langle{#1}\rangle}}

\newcommand{\Orth}{\mathop{\null\mathrm {O}}\nolimits}

\renewcommand{\Tilde}{\widetilde}

\newcommand{\aaa}{\mathbb A}
\newcommand{\ddd}{\mathbb D}
\newcommand{\eee}{\mathbb E}

\begin{document}
\title{Examples
of lattice-polarized K3 surfaces with automorphic discriminant,
and Lorentzian Kac--Moody algebras}
\date{24 February 2017}
\author{Valery  Gritsenko\footnote{The first author was  supported
by Laboratory of Mirror Symmetry NRU HSE, RF government grant, ag. 
N 14.641.31.0001
and Institut Universitaire de France (IUF).}\ \ and  Viacheslav V. Nikulin}

\maketitle

\begin{abstract}
Using our results about Lorentzian Kac--Moody algebras and arithmetic mirror
symmetry, we give six series of examples of lattice-polarized K3 surfaces 
with automorphic discriminant.
\end{abstract}

\centerline{Dedicated to \'E.B. Vinberg on the occasion of 
his 80th Birthday}

\section{Introduction}
\label{introduction}
Using results of our recent paper \cite{GN9} and our previouse papers,
we construct many even hyperbolic lattices $S$ such that $S$-polarized
complex K3 surfaces $X$ have an automorphic discriminant.

We recall that for a $S$-polarized
K3 surface $X$ a primitive embedding $S\subset S_X$ is fixed where $S_X$ is
the Picard lattice of $X$.
We say that such $X$ is degenerate (or it belongs to the discriminant)  
if there exists $\delta \in (S)^\perp_{S_X}$
such that $\delta^2=-2$. By geometry of K3 surfaces, 
it then follows that $X$ has no
a polarization $h$ from $S$.
By Global Torelli Theorem \cite{PS} and epimorphicity of 
period map for K3
surfaces \cite{Kul}, moduli of such K3 surfaces are covered 
by the corresponding
hermitian symmetric domains, and algebraic functions on moduli 
are the corresponding
automorphic forms on these domains.
A holomorphic automorphic form is called {\it discriminant}
if the support of its zero divisor is  equal to
the preimage of the discriminant of moduli of
such K3 surfaces. If a discriminant automorphic form
exists, the discriminant is then called {\it automorphic}.

For example, for $S=\bz h$ of the rank one
with $h^2=n$ where $n\ge 2$ is even (that is for usual 
polarized K3 surfaces), it
is well-known that the discriminant automorphic form exists for $n=2$.
Borcherds constructed the discriminant automorphic form for $n=2$ explicitly
(see \cite[pp. 200--201]{B2}).
It was shown in \cite{Nik9} that for infinite number
of even $n\ge 2$ the discriminant automorphic
form does not exist (probably, it was the first result in this direction).
Later, Looijenga \cite{Loo} showed that the discriminant
automorphic form does not exist and the discriminant is not automorphic
for all $n>2$.

Here, we ask about examples of automorphic discriminants for $\rk S\ge 2$.
See some finiteness results in Ma \cite{Ma}.

In Sect. \ref{section1}, we give necessary definitions about $S$-polarized
K3 surfaces and their discriminants and automorphic discriminants.

In Sect. \ref{section2}, we formulate the main Theorems
\ref{th:autdiscr1} and \ref{th:autdiscr2}
which give six series of even hyperbolic lattices $S$ of $\rk S\ge 2$
such that $S$-polarized K3 surfaces have an automorphic discriminant.
They are given in Tables 1---6. All these examples are related to
the Lorentzian Kac-Moody algebras  constructed in \cite{GN9}, those are
hyperbolic automorphic Kac--Moody (super) Lie algebras.
The corresponding discriminant automorphic forms are given in \cite{GN9}
and define such Kac--Moody algebras $\gggg$, and give their 
denominator identities.

It would be interesting to understand geometric meaning of these
automorphic forms and Kac--Moody algebras for the geometry of the 
corresponding K3 surfaces. For example,  we know that if the weight 
of the discriminant automorphic form is larger than the dimension of 
the moduli space then the moduli space is at least uniruled
(see Theorem 3.4 in \S 3).

\section{Lattice-polarized K3 surfaces and their\\ moduli and  discriminants}
\label{section1}

We refer to \cite{Nik1} about lattices.
We recall that a lattice $M$ (equivalently, a non-degenerate
integral symmetric bilinear form) means that $M$ is a free $\bz$-module 
$M$ of a finite rank
with symmetric $\bz$-bilinear non-degenerate pairing $x\cdot y\in \bz$ for
$x,\,y\in M$. By signature of $M$, we mean the signature
of the corresponding real form $M\otimes \br$ over $\br$ (that is the numbers
$(t_{(+)},\,t_{(-)})$ of positive and negative squares respectively).
A lattice $M$ of the signature
$(1, \rk M-1)$ is called {\it hyperbolic.} A lattice $M$ is called {\it even} if
$x^2=x\cdot x $ is even for any $x\in M$. By $O(M)$, we denote the automorphism
group of a lattice $M$. Each element $\delta\in M$ with $0\not=\delta^2$
and $\delta^2| 2(\delta\cdot M)$ (it is called {\it root}) defines the reflection
$s_\delta:x\mapsto x-[(2(x\cdot \delta)/\delta^2]\delta$ for $x\in M$. Evidently,
$s_\delta \in O(M)$, $s_\delta(\delta)=-\delta$ and $s_\delta$ is 
identical on $\delta^\perp_M$.
By $W^{(2)}(M)\subset O(M)$, we denote the subgroup generated by reflections
in all elements $\delta\in M$ with $\delta^2=-2$ (they are all roots).

Let $S$ be a hyperbolic lattice. Let
$$
V(S)=\{x\in S\otimes \br\ |\ x^2>0\}
$$
be the {\it cone of} $S$. It has two connected components $V^+(S)$ and
$V^-(S)=-V^+(S)$. We fix one of them, $V^+(S)$, and the corresponding
hyperbolic space $\La(S)=V^+(S)/\br_{++}$. Here $\br_{++}$ denotes
all positive real numbers and $\br_+$ denotes all non-negative real numbers.
Let ${\rm Amp}(S)$, ${\rm Amp}(S)/\br_{++}$ be interiors  of 
fundamental chambers
for the reflection group $W^{(2)}(S)$ in $V^+(S)$ and $\La (S)$ respectively.
We fix one of them. Thus,
we fix the pair $(V^+(S), {\rm Amp}(S))$. They are defined uniquely up
to the action of $O(S)$. We call the pair as the {\it ample cone of} $S$.
It is equivalent to ${\rm Amp}(S)$ or ${\rm Amp}(S)/\br_{++}$.

Let $X$ be a K\"alerian K3 surface
(for example, see \cite{PS}, \cite{Kul}, \cite{BR}, \cite{Tod}, \cite{Siu}
about such surfaces), that is $X$ is a non-singular compact
complex surface with trivial canonical class
$K_X$ (equivalently, $0\not=\omega_X\in H^{2,0}(X)=\Omega^2[X]$
has the zero divisor) and such that the irregularity $q(X)$ is equal to $0$
(equivalently, $X$ has no
non-zero holomorphic $1$-dimensional holomorphic forms).
Then $H^{2,0}(X)=\bc \omega_X$ and $H^2(X,\bz)$ with the intersection pairing
is an even unimodular (that is with the determinant $\pm 1$) lattice 
$L_{K3}$ of signature $(3,19)$.
The primitive sublattice
$$
S_X=H^2(X,\bz)\cap H^{1,1}(X)=\{x\in H^2(X,\bz)\ |\ x\cdot \omega_X=0\}
\subset H^2(X,\bz)
$$ is the {\it Picard lattice}
of $X$ generated by the first Chern classes of all line bundles over $X$.
Here {\it primitive} means that $H^2(X,\bz)/S_X$ has no torsion.
By the difinition, $S_X$ can be either negative definite,
semi-negative definite, or hyperbolic lattice.
By Kodaira, the last case is exactly
the case when $X$ is projective algebraic.

Further, we assume that $X$ is algebraic. We denote by $V^+(S_X)=V(X)$ the
half cone of $S_X$ which contains a polarization of $X$, and by
${\rm Amp} (X)\subset V(X)$ the ample cone of $X$.
Then ${\rm Amp} (S_X)={\rm Amp} (X)$ gives an ample cone of $S_X$, 
see \cite{PS}.

Further, we fix an even hyperbolic lattice $S$ and its ample cone 
${\rm Amp} (S)$.

We remind to a reader (e. g. see \cite{DN},
\cite{D1} and \cite{Nik0})
that a {\it $S$-polarized K3 surface $X$}
means that a primitive embedding
$S\subset S_X$ of lattices is fixed such that
${\rm Amp}(S)\cap {\rm Amp}(X)\not=\emptyset$.

If instead of the last condition only the conditions
${\rm Amp}(S)\cap \overline{{\rm Amp}(X)}\not=\emptyset$
and ${\rm Amp}(S)\cap {\rm Amp}(X)=\emptyset$ are satisfied,
then we say that {\it $X$ is a degenerate $S$-polarized K3
surface, equivalently, $X$ belongs to the discriminant of moduli
of $S$-polarized K3 surfaces.} By geometry of K3 surfaces (see \cite{PS}),
it happens only if there exists $\delta \in (S)^\perp_{S_X}$ such
that $\delta^2=-2$.

By global Torelli Theorem for K3 surfaces \cite{PS} and empimorphicity of
period map for K3 surfaces \cite{Kul}, for general $S$-polarized K3 surfaces
we have $S_X=S$ and ${\rm Amp}(X)={\rm Amp}(S)$, for non-degenerate 
$S$-polarized
K3 surfaces $X$, the  $(S)^\perp_{S_X}$ has no
elements $\delta$ with $\delta^2=-2$ and 
${\rm Amp}(X)\cap {\rm Amp}(S)\not=\emptyset$,
and for degenerate $S$-polarized K3 surfaces $X$, the  
$(S)^\perp_{S_X}$ has
elements $\delta$ with $\delta^2=-2$ and only
${\rm Amp}(S)\cap \overline{{\rm Amp}(X)}\not=\emptyset$
is valid, {\it equivalently, $X$ belongs to the discriminant of
moduli of $S$-polarized K3 surfaces.}

For a $S$-polarized K3 surface $X$, let us consider periods
$$
H^{2,0}(X)=\bc \omega_X\subset T_X\otimes \bc\subset T\otimes \bc
$$
where $T_X=(S_X)^\perp_{H^2(X,\bz)}$ is the transcendental 
lattice of $X$ and
$T=(S)^\perp_{H^2(X,\bz)}$ is the transcendental lattice
of the $S$-polarization. The periods give a point in IV type 
Hermitian symmetric
domain
$$
\Omega(T)=\{\bc \omega \subset T\otimes \bc\ |\
\omega\cdot \omega =0\ and\ \omega\cdot \overline{\omega}>0\}^+
$$
where $+$ means a choice of one of two connected components.
This point belongs to the complement of
the discriminant
\begin{equation}
{\rm Discr} (T)=\bigcup_{\beta\in T^{(2)}}{D_\beta},
\label{discr}
\end{equation}
where
$$
D_\beta=\{\bc\omega\in \Omega(T)\ |\ \omega\cdot \beta=0\}
$$
is the rational quadratic divisor which is orthogonal to 
$\beta\in T$ with $\beta^2<0$;
we recall that $\beta^2=-2$ for $\beta\in T^{(2)}$. Of course, 
$D_\beta=D_{-\beta}$,
and we identify $\pm \beta$ in this definition. Further,
$$
O^+(T)=\{g\in O(T)\ |\ g(\Omega(T))=\Omega(T)\}
$$
is the group of automorphisms of $T$ which preserve the connected
component $\Omega(T)$.

By considering all possible isomorphism classes
$T_1,\dots,T_n$ of the transcendental lattice $T$ for all primitive embeddings
$S\subset L_{K3}$, we correspond to a $S$-polarized K3 surface $X$ a point in
\begin{equation}
Mod(S)=\bigcup_{1\le k\le n}{G_k\backslash(\Omega(T_k)-{\rm Discr}(T_k))}
\label{autdisc}
\end{equation}
where $G_k\subset O^+(T_k)$ is an appropriate finite index subgroup. By
global Torelli Theorem \cite{PS} and epimorphicity of the period map for K3
surfaces \cite{Kul}, each point of $Mod(S)$ corresponds to some 
$S$-polarized K3 surface $X$.

We recall that a holomorphic function $\Phi$ on $\Omega (T)$ is 
called an {\it automorphic form of
a weight $d\in \mathbb N$} if $\Phi$ is homogeneous of the degree 
$(-d)$ and it is symmetric
with respect to a subgroup $H\subset O^+(T)$ of finite index.

Finally, we can give a defintion:

\begin{definition} We fix an even hyperbolic lattice $S$.
We say that $S$-polarized K3 surfaces
have  an {\bf automorphic discriminant}
if for each $1\le k \le n$ in \eqref{autdisc} there exists a holomorphic
automorphic form on $\Omega(T_k)$ such that the support of its zero
divisor is equal to $Discr(T_k)$ in \eqref{discr}. Then we call this
automorphic form {\bf discriminant automorphic form.}
\label{defautdisc}
\end{definition}

The {\it stable orthogonal group}
$$
\widetilde{O}^+(T)=\{g\in O^+(T)\ |\ g|_{T^\ast/T}={\rm id}\}
$$
is a subgroup of finite index of $O^+(T)$.
For a primitive embedding $S\subset L_{K3}$ and $T=(S)^\perp_{L_{K3}}$,
the group $\widetilde{O}^+(T)$ consists of automorphisms from $O^+(T)$
which can be continued to an element of $O(L_{K3})$ identically on $S$.
Thus, we can assume that $\widetilde{O}^+(T_k)\subset G_k$.

%%%%%%%%%%%%%%%%%%%%%%%%%%%%%%%%%%%%%%%%%%%%%%%%%%%%%
%%%%%%%%%%%%%%%%%%%%%%%%%%%%%%%%%%%%%%%%%%%%%%%%%%%%%%
%%%%%%%%%%%%%%%%%%%%%%%%%%%%%%%%%%%%%%%%%%%%%%%%%%%

\section{Lattice-polarized K3 surfaces with\\ automorphic discriminant
related to\\ Lorentzian Kac--Moody algebras with Weyl
groups of $2$-reflections}
\label{section2}

Below, we use the following notations for lattices. We use $\oplus$ 
for the orthogonal
sum of lattices. By $tM$, we denote the orthogonal sum of $t$ copies of $M$.
By $A_k$, $k\ge 1$, $D_m$, $m\ge 4$, $E_l$, $l=6,7,8$, we denote the
standard root lattices with Dynkin diagrams $\aaa_k$, $\ddd_m$, 
$\eee_l$ respectively
and the roots with square $(-2)$. For a lattice $M$, we denote by 
$M(t)$ the lattice which
is obtained from $M$ by multiplication by $0\not=t\in \bq$ of the 
bilinear form of the lattice $M$
if the form of $M(t)$ remains integral.
By $\langle A \rangle$, we denote the lattice with the symmetric 
matrix $A$. We denote
by
\begin{equation}
U= \left\langle
\begin{array}{cc}
0 & 1 \\
1  & 0
\end{array}
\right\rangle
\label{U}
\end{equation}
the even unimodular lattice of signature $(1,1)$. For example, 
$L_{K3}\cong 3U\oplus 2E_8$.

We remind to a reader that for an
integer lattice $M$ we have the canonical embedding $M\subset M^\ast=Hom(M,\bz)$.
It defines a (finite) discrminant group $A_M=M^\ast/M$.
By continuing the symmetric bilinear form of the lattice $M$ to $M^\ast$,
we obtain a finite symmetric bilinear form $b_M$ on $A_M$ with 
values in $\bq/\bz$
and a finite quadratic form $q_M$ on $A_M$ with values in $\bq/2\bz$, 
if $M$ is even.
They are called the {\it discriminant forms of the lattice $M$.}

If there are no other conditions, by $(M)^\perp_L$ we mean
an orthogonal complement to a lattice $M$ in a lattice $L$
for some primitive embedding $M\subset L$. For the most cases of
the Theorems \ref{th:autdiscr1} and \ref{th:autdiscr2} below,
the orthogonal complement is unique up to isomorphism. For other case, it
does not matter which isomorphism class we shall take.

We have the following six series of examples of even hyperbolic lattices
$S$ of $\rk S\ge 2$ such that $S$-polarized K3 surfaces have an 
automorphic discriminant.
They are given in the theorems below.

\begin{theorem}
For the hyperbolic lattices $S$ which are given in the last columns of
the Tables 1---6 below, $S$-polarized K3 surfaces have an automorphic
discriminant. We also give the discriminant quadratic form $q_S$ of $S$
in notations of \cite{CS}. The even hyperbolic lattice $S$ is 
defined by its rank and $q_S$
uniquely up to isomorphism (see proofs below).

For all these cases, the transcendental lattice $T=(S)^\perp_{L_{K3}}$,
$L_{K3}=3U\oplus 2E_8$, is
unique up to isomorphism, and its isomorphism class is equal to
$T=U(m)\oplus S^{mir}$ where the hyperbolic lattice $S^{mir}$ is shown
in the first column and $m$ is shown in the second column of the table
in the same line as $S$.

\label{th:autdiscr1}
\end{theorem}

\begin{theorem}
For all cases of Theorem \ref{th:autdiscr1}, the discriminant automorphic form
$\Phi(z)$ has the  Fourier expansion with integral coefficients at
the zero dimensional cusp defined by the decomposition 
$T=U(m)\oplus S^{mir}$ (see \cite{GN9}),
$z\in S^{mir}\otimes \br+\sqrt{-1}\,V^+(S^{mir})$.
The Fourier coefficients  define a Loren\-tzian (hyperbolic and
automorphic) Kac--Moody super-algebra $\gggg$  which is 
graded by the hyperbolic
lattice $S^{mir}$. The $\Phi(z)$ has an infinite product 
(Borcherds) expansion which
gives multiplicities of roots of this algebra. 
See \cite{K1}, \cite{K2}, \cite{B1},
\cite{B2}.

The divisor of $\Phi(z)$ is sum of
rational quadratic divisors $D_\alpha$, $\alpha\in T^{(2)}$,
with multiplicities one.

The $S^{mir}$-polarized K3 surfaces can be considered as 
{\it mirror symmetric}
to $S$-polarized K3 surfaces by mirror symmetry considered 
in \cite{DN}, \cite{D1},
\cite{GN3}, \cite{GN7}. They have the remarkable property 
that there exists
$\rho\in S^{mir}\otimes \bq$ such that $\rho\cdot E=1$ for 
each non-singular rational
curve $E\subset X$ with $S_X=S^{mir}$ (for $\rho^2>0$ and 
$\rk S^{mir}=4$, such $S^{mir}$
are in the list of $14$ lattices which were found 
by \'E.B. Vinberg in \cite{Vin};
about other $S^{mir}$ see  \cite{Nik2} and \cite{Nik5}).

\label{th:autdiscr2}
\end{theorem}

\begin{proof} Theorems \ref{th:autdiscr1} and \ref{th:autdiscr2}
are mainly reformulations of the results of \cite{GN9} using
the discriminants forms technique for integer lattices
which was developed in \cite{Nik1}.

Let $S$ be a lattice of one of Tables 1---6. By results 
of \cite{Nik1},
we have $T=(S)^\perp_{3U\oplus 2E_8}\cong U(m)\oplus S^{mir}$ 
where $S^{mir}$
and $m$ are shown in the same line of the table as $S$. Here, 
it is important
that the discriminant quadratic forms $q_T$ and $q_S$ are 
related as $q_T\cong -q_S$
since $T\perp S$ in the unimodular lattice $3U\oplus 2E_8$. 
Vice a versa,
$(S)^\perp_{3U\oplus 2E_8}=T$ and $(T)^\perp_{3U\oplus 2E_8}=S$ 
for some primitive
ebmeddings $S\subset 3U\oplus 2E_8$ and $T\subset 3U\oplus 2E_8$ if
signatures of $T$, $S$ and $3U\oplus 2E_8$ agree and $q_T\cong -q_S$;
the signature $(t_{(+)}, t_{(-)})$ together with the discriminant 
quadratic form $q$
define the genus of an even lattice; Theorem \cite[Theorem 1.13.1]{Nik1}
(which uses results by M. Kneser) gives conditions when an even
indefinite lattice with the invariants $(t_{(+)}, t_{(-)}, q)$
is unique up to isomorphism.

For all $S^{mir}$ and $m$ which are shown in lines of Tables 1---6,
the automorphic form $\Phi (z)$ with the properties mentioned 
in the Theorems
\ref{th:autdiscr1} and \ref{th:autdiscr2}
is constructed in \cite{GN9}.
For the lattices of Table 1, it is done in  
\cite[Theorem 4.3 and Proposition 4.1]{GN9};
of Table 2, in \cite[Theorem 4.4]{GN9};
of Table 3, in \cite[Theorem 6.1]{GN9};
of Table 4, in \cite[Theorems 6.2 and 6.3]{GN9};
of Table 5, in \cite[Lemma 6.4]{GN9};
of Table 6, in \cite[Theorem 6.5]{GN9}.

By results of \cite{Nik1} which were mentioned above,
we have $S=(T)^\perp_{3U\oplus E_8}$ is unique up to 
isomorphism, and $S$ is
shown in the Tables.

These considerations give the proof.
\end{proof}

\begin{table}
\label{table1}
\caption{$S$-polarized K3 surfaces
with automorphic discriminant.}

\medskip

%%\index{Table 1}

%%%\addtocontents{toc}{\contentsline {section}{\tocsection {}{T.1}{Table 1}}
%%%{\pageref{table1}}}

\begin{tabular}{|c|c|c|c|c|c|c|c|c|}
\hline
$S^{mir}$& $T=$         & weight       &  $S=(T)^\perp_{3U\oplus 2E_8}$& $q_S$\\
         & $U(m)\oplus S^{mir}$ &of $\Phi(z)$ &       &                  \\
\hline
\hline
 $U\oplus A_1$  &$m=1$ & $35$ &   $U\oplus E_8\oplus E_7$  & $2_1^{+1}$  \\
\hline
 $U\oplus 2A_1$  &$m=1$ & $34$ &   $U\oplus E_8\oplus D_6$ & $2_2^{+2}$ \\
\hline
 $U\oplus A_2$  &$m=1$ & $45$ &   $U\oplus E_8\oplus E_6$  & $3^{-1}$ \\
\hline
 $U\oplus 3A_1$  &$m=1$ & $33$ &   $U\oplus E_7\oplus D_6$ & $2_3^{+3}$  \\
\hline
 $U\oplus A_3$  &$m=1$ & $54$ &   $U\oplus E_8\oplus D_5$  & $4_3^{-1}$  \\
\hline
 $U\oplus 4A_1$  &$m=1$ & $32$ &   $U\oplus D_6\oplus D_6$ & $2_4^{+4}$  \\
\hline
 $U\oplus 2A_2$  &$m=1$ & $42$ &   $U\oplus E_6\oplus E_6$ & $3^{+2}$ \\
\hline
 $U\oplus A_4$  &$m=1$ & $62$ &   $U\oplus E_8\oplus A_4$  & $5^{+1}$ \\
\hline
 $U\oplus D_4$  &$m=1$ & $72$ &   $U\oplus E_8\oplus D_4$  & $2_{II}^{-2}$ \\
\hline
 $U\oplus D_4$  &$m=2$ & $40$ &   $U(2)\oplus E_8\oplus D_4$  & $2_{II}^{-4}$\\
\hline
 $U\oplus A_5$  &$m=1$ & $69$ &   $U\oplus E_8\oplus A_2\oplus A_1$  & 
 $2_{-1}^{+1},3^{+1}$ \\
\hline
 $U\oplus D_5$  &$m=1$ & $88$ &   $U\oplus E_8\oplus A_3$  & $4_5^{-1}$ \\
\hline
 $U\oplus 3A_2$  &$m=1$ & $39$ &   $U\oplus E_6\oplus 2A_2$ & $3^{-3}$  \\
\hline
 $U\oplus 2A_3$  &$m=1$ & $48$ &   $U\oplus 2D_5$  & $4_6^{+2}$ \\
 \hline
 $U\oplus A_6$  &$m=1$ & $75$ &   $U\oplus E_8\oplus
 \left\langle
\begin{array}{cc}
-2 & 1 \\
1  & -4
\end{array}
\right\rangle$   & $7^{-1}$ \\
\hline
 $U\oplus D_6$  &$m=1$ & $102$ &   $U\oplus E_8\oplus 2A_1$  & $2_{-2}^{+2}$\\
\hline
 $U\oplus E_6$  &$m=1$ & $120$ &   $U\oplus E_8\oplus A_2$   & $3^{+1}$ \\
\hline
 $U\oplus A_7$  &$m=1$ & $80$ &   $U\oplus E_8\oplus \langle -8 \rangle$   
 & $8_{-1}^{+1}$ \\
\hline
 $U\oplus D_7$  &$m=1$ & $114$ &   $U\oplus E_8\oplus \langle -4 \rangle$  
 & $4_{-1}^{+1}$ \\
\hline
 $U\oplus E_7$  &$m=1$ & $165$ &   $U\oplus E_8\oplus A_1$  & $2_{-1}^{+1}$ \\
\hline
 $U\oplus 2D_4$  &$m=1$ & $60$ &   $U\oplus 2D_4$  &  $2_{II}^{+4}$\\
\hline
 $U\oplus D_8$  &$m=1$ & $124$ &   $U\oplus D_8$  & $2_{II}^{+2}$ \\
\hline
 $U\oplus E_8$  &$m=1$ & $252$ &   $U\oplus E_8$  & $0$ \\
 \hline
 $U(2)\oplus 2D_4$  &$m=1$ & $28$ &   $U(2)\oplus 2D_4$  & $2_{II}^{+6}$  \\
\hline
 $U\oplus 2E_8$  &$m=1$ & $132$ &   $U$  & $0$  \\
\hline

\end{tabular}
\end{table}

\begin{table}
\label{table2}
\caption{$S$-polarized K3 surfaces
with automorphic discriminant.}

\medskip

%%\index{Table 2}

%%%\addtocontents{toc}{\contentsline {section}{\tocsection {}{T.1}{Table 1}}
%%%{\pageref{table1}}}

\begin{tabular}{|c|c|c|c|c|c|c|c|c|}
\hline
$S^{mir}$& $T=$         & weight       &  $S=(T)^\perp_{3U\oplus 2E_8}$& $q_S$\\
         & $U(m)\oplus S^{mir}$ &of $\Phi(z)$ &       &                   \\
\hline
\hline
 $U$  &$m=1$ & $12$ &   $U\oplus E_8\oplus E_8$  & $0$  \\
\hline
 $U\oplus A_1(2)$  &$m=1$ & $12$ &   $U\oplus E_8\oplus D_7$  & $4_1^{+1}$  \\
\hline
 $U\oplus A_1(3)$  &$m=1$ & $12$ &   $U\oplus E_8\oplus E_6\oplus A_1$  & 
 $2_{-1}^{+1},3^{-1}$ \\
\hline
 $U\oplus A_1(4)$  &$m=1$ & $12$ &   $U\oplus E_8\oplus A_7$   & $8_1^{+1}$\\
\hline
 $U\oplus 2\langle -4\rangle$  &$m=1$ & $12$ &   $U\oplus D_7\oplus D_7$  & 
 $4_2^{+2}$    \\
\hline
 $U\oplus A_2(2)$  &$m=1$ & $12$ &   $U\oplus E_8\oplus D_4\oplus A_2$   & 
 $2_{II}^{-2},3^{+1}$   \\
 \hline
 $U\oplus A_2(3)$  &$m=1$ & $12$ &   $U\oplus E_8\oplus (A_2(3))^\perp_{E_8}$& 
 $3^{-1},9^{-1}$   \\
 \hline
 $U\oplus A_3(2)$  &$m=1$ & $12$ &   $U\oplus E_8\oplus (A_3(2))^\perp_{E_8}$ & 
 $2_{II}^{-2},8_3^{-1}$ \\
 \hline
 $U\oplus D_4(2)$  &$m=1$ & $12$ &   $U\oplus E_8\oplus D_4(2)$  
 & $2_{II}^{-2},4_{II}^{-2}$  \\
 \hline
$U\oplus E_8(2)$  &$m=1$ & $12$ &   $U\oplus E_8(2)$   &  $2_{II}^{+8}$ \\
 \hline

\end{tabular}
\end{table}

\begin{table}
\label{table3}
\caption{$S$-polarized K3 surfaces
with automorphic discriminant.}

\medskip

%%\index{Table 3}

%%%\addtocontents{toc}{\contentsline {section}{\tocsection {}{T.1}{Table 1}}
%%%{\pageref{table3}}}

\begin{tabular}{|c|c|c|c|c|c|c|c|c|}
\hline
$S^{mir}$& $T=$         & weight       &  $S=(T)^\perp_{3U\oplus 2E_8}$& $q_S$\\
         & $U(m)\oplus S^{mir}$ &of $\Phi(z)$ &       &                  \\
\hline
\hline
 $\langle 2 \rangle \oplus A_1$  &$m=2$ & $12$ &   
 $U(2)\oplus E_8\oplus E_7\oplus A_1$ & $2_0^{+4}$ \\
 \hline
 $\langle 2 \rangle \oplus 2A_1$  &$m=2$ & $11$ &   
 $U(2)\oplus E_7\oplus E_7\oplus A_1$ &$2_1^{+5}$ \\
 \hline
 $\langle 2 \rangle \oplus 3A_1$  &$m=2$ & $10$ &   
 $U(2)\oplus E_7\oplus D_6\oplus A_1$ & $2_2^{+6}$\\
 \hline
$\langle 2 \rangle \oplus 4A_1$  &$m=2$ & $9$ &   
$U(2)\oplus D_6\oplus D_6\oplus A_1$   & $2_3^{+7}$\\
 \hline
$\langle 2 \rangle \oplus 5A_1$  &$m=2$ & $8$ &   
$U\oplus D_6\oplus 6A_1$      & $2_4^{+8}$\\
 \hline
$\langle 2 \rangle \oplus 6A_1$  &$m=2$ & $7$ &   
$U(2)\oplus D_6\oplus 5A_1$   & $2_5^{+9}$\\
 \hline
$\langle 2 \rangle \oplus 7A_1$  &$m=2$ & $6$ &   
$U(2)\oplus D_4\oplus 6A_1$   & $2_6^{+10}$\\
 \hline
$\langle 2 \rangle \oplus 8A_1$  &$m=2$ & $5$ &   
$U(2)\oplus E_8(2)\oplus A_1$ & $2_7^{+11}$\\
 \hline

\end{tabular}
\end{table}

%%%%%%%%%%%%%%%%%%%%%%%%%%%%%%%%%%%%%%%%%%%%%
%%%%%%%%%%%%%%%%%%%%%%%%%%%%%%%%%%%%%%%%%%%%%%%%

\begin{table}
\label{table4}
\caption{$S$-polarized K3 surfaces
with automorphic discriminant.}

\medskip

%%\index{Table 4}

%%%\addtocontents{toc}{\contentsline 
%%%{section}{\tocsection {}{T.1}{Table 3}}
%%%{\pageref{table4}}}

\begin{tabular}{|c|c|c|c|c|c|c|c|c|}
\hline
$S^{mir}$& $T=$         & weight       &  
$S=(T)^\perp_{3U\oplus 2E_8}$   & $q_S$  \\
         & $U(m)\oplus S^{mir}$ &of $\Phi(z)$ &       &    \\
\hline
\hline
 $U(2)\oplus D_4$  &$m=1$ & $40$ &   $U(2)\oplus E_8\oplus D_4$   & 
 $2_{II}^{-4}$ \\
\hline
 $U(2)\oplus D_4$  &$m=2$ & $24$ &   $U\oplus 3D_4$ & $2_{II}^{-6}$ \\
 \hline
 $U(4)\oplus D_4$  &$m=4$ & $6$ &   $U(4)\oplus (U(4)\oplus D_4)^\perp_{U\oplus 2E_8}$ 
 & $2_{II}^{-2},4_{II}^{+4}$ \\
 \hline

\end{tabular}
\end{table}

%%%%%%%%%%%%%%%%%%%%%%%%%%%%%%%%%%%%%%%%%%
%%%%%%%%%%%%%%%%%%%%%%%%%%%%%%%%%%%%%%%%%%

\begin{table}
\label{table5}
\caption{$S$-polarized K3 surfaces
with automorphic discriminant.}

\medskip

%%\index{Table 5}

%%%\addtocontents{toc}{\contentsline {section}{\tocsection {}{T.1}{Table 5}}
%%%{\pageref{table3}}}

\begin{tabular}{|c|c|c|c|c|c|c|c|c|}
\hline
$S^{mir}$& $T=U(m)$         & weight       &  
$S=(T)^\perp_{3U\oplus 2E_8}$   & $q_S$  \\
         & $\oplus S^{mir}$ &of $\Phi(z)$ &       &                      \\
\hline
\hline
 $U(4)\oplus A_1$  &$m=4$ & $5$ &   
 $U(4)\oplus (U(4))^\perp_{U\oplus E_8}\oplus E_7$ & $2_1^{+1},4_{II}^{+4}$\\
 \hline
 $U(4)\oplus 2A_1$  &$m=4$ & $4$ &   
 $U(4)\oplus (U(4))^\perp_{U\oplus E_8}\oplus D_6$ & 
 $2_2^{+2},4_{II}^{+4}$ \\
 \hline
 $U(4)\oplus 3A_1$  &$m=4$ & $3$ &   
 $U(4)\oplus (U(4))^\perp_{U\oplus E_8}\oplus $
 &  $2_3^{+3},4_{II}^{+4}$\\ 
                    &      &      &   
                    $\oplus D_4\oplus A_1$    & \\                
 \hline
 $U(4)\oplus 4A_1$  &$m=4$ & $2$ &   
 $U(4)\oplus (U(4))^\perp_{U\oplus E_8}\oplus 4A_1$
 &  $2_4^{+4},4_{II}^{+4}$\\

 \hline

\end{tabular}
\end{table}

%%%%%%%%%%%%%%%%%%%%%%%%%%%%%%%%%%%%%%%%%%%%%%%%%%%%%%%
%%%%%%%%%%%%%%%%%%%%%%%%%%%%%%%%%%%%%%%%%%%%%%%%%%%%%%%%
%%%%%%%%%%%%%%%%%%%%%%%%%%%%%%%%%%%%%%%%%%%%%%%%%%%%%%
\begin{table}
\label{table6}
\caption{$S$-polarized K3 surfaces
with automorphic discriminant.}

\medskip

%%\index{Table 6}

%%%\addtocontents{toc}{\contentsline {section}{\tocsection {}{T.1}{Table 6}}
%%%{\pageref{table3}}}

\begin{tabular}{|c|c|c|c|c|c|c|c|c|}
\hline
$S^{mir}$& $T=U(m)$         & weight       &  
$S=(T)^\perp_{3U\oplus 2E_8}$   & $q_S$  \\
         & $\oplus S^{mir}$ &of $\Phi(z)$ &       &        \\
\hline
\hline
 $U(3)\oplus A_2$  &$m=3$ & $9$ &   
 $U(3)\oplus (U(3))^\perp_{U\oplus E_8}\oplus E_6$ & $3^{-5}$\\
 \hline
 $U(3)\oplus 2A_2$  &$m=3$ & $6$ &   
 $U(3)\oplus (U(3))^\perp_{U\oplus E_8}\oplus 2A_2$ &$3^{+6}$ \\
 \hline
 $U(3)\oplus 3A_2$  &$m=3$ & $3$ &   
 $U(3)\oplus (U(3)\oplus A_2)^\perp_{U\oplus E_8}\oplus 2A_2$ & $3^{-7}$\\
 \hline

\end{tabular}
\end{table}

\newpage

In many cases, existence of the  automorphic 
discriminant tells us that the moduli
space of the corresponding $S$-polarized  
K3 surfaces has a special geometry.
The following criterion is valid.

\begin{theorem} (See \cite[Theorem 2.1]{GH}.)
Let $\Omega(T)$ be a connected component of the type $IV$ 
domain associated to a
lattice $T$ of signature $(2,n)$ with $n \geq 3$ and 
let $\Gamma \subset \Orth^+(T)$
be an arithmetic subgroup of finite index of the orthogonal group.  Let
$\widetilde{B}= \sum_{r} D_{r}$ in $\Omega(T)$ be the divisorial part of the
ramification locus of the quotient map 
$ \Omega(T) \to \Gamma \backslash \Omega(T)$.
(This means that the reflection $s_r$ or $-s_r$ belongs to $\Gamma$).
Assume that a modular form $F_k$
with respect to $\Gamma$ of weight $k$ with a 
(finite order) character exists, such
that
$$
\{F_k=0\} = \sum_{r} m_{r} D_{r}
$$
where the $m_{r}$ are non-negative integers.
Let $m = \max\{m_{r}\}$ (which must be $> 0$ by Koecher's principle).
If $k >  m\cdot n$,
then $\Gamma' \backslash \cD$ is uniruled for every 
arithmetic group $\Gamma'$
containing
$\Gamma$.
\end{theorem}

Using this criterion, we prove

\begin{theorem}\label{thm-uniruled}
The moduli space of $S$-polarized K3 surfaces is at 
least uni\-ruled if
$S$ is any lattice of Table $1$ and Table $2$, a lattice 
from the first five lines
of Table $3$ (till the lattice  $\latt{2}\oplus 5A_1$),  
the first two lines of Tables
$4$  and the first line of Tables  $5$ and $6$.
\end{theorem}

\begin{proof}
The moduli space of $S$-polarized K3 surfaces is 
defined in (\ref{autdisc}).
For any lattice $S$ in Tables $1$--$6$ there is 
only one isomorphism class of the
corresponding lattices $T$, i.e. there is only 
one term in (\ref{autdisc}).
The modular group $G=G_1$ of the moduli space always 
contains the stable orthogonal
group $\Tilde O^+(T)$ acting trivially on the discriminant 
quadratic form of $T$.
The divisor $D_r$ with $r^2=-2$, $r\in T$, always belongs 
to the ramification divisor
since $s_r\in \Tilde O^+(T)$. We note that   $\Tilde O^+(T)$ 
is generated
by $-2$-reflections for the most part of the lattices from 
Tables $1$ and $2$
(see \cite{GHS}).
By  construction (see \cite[\S 4]{GN9}), any
discriminant automorphic form from  Tables $1$
and $2$ is a modular form with respect to $\Tilde O^+(T)$ 
with character det with the
simplest possible divisor ${\rm Discr}(T)$
of multiplicity one.  The weight of the discriminant 
automorphic form is shown in the
Tables. If the dimension $n$ of the moduli space is larger 
than $2$, we apply Theorem
3.3. If $n=1$ or $2$, the corresponding modular varieties 
are at least unirational.

The construction of the discriminant automorphic forms of 
Table 3 uses the isomorphism
$$
O(U(2)\oplus (\latt{2}\oplus (k+1)\latt{-2}))\cong
O(U\oplus (\latt{1}\oplus (k+1)\latt{-1}))\cong O(U\oplus U\oplus D_k)
$$
(see \cite[Lemma 6.1]{GN9}). Moreover,
the reflections with respect to  $-2$-vectors of 
$\latt{2}\oplus (k+1)\latt{-2}$
correspond to the reflections with respect to $-4$-reflective vectors of
$U\oplus D_k $ or $-1$-vectors of $U\oplus D_k^* $.
If $k\ne 4$,  then  all $-1$-reflective vectors
of $2U\oplus D_k^* $ belong to the unique 
$\Tilde\Orth^+(2U\oplus D_k )$-orbit which
is equal
to the set of  $-1$-vectors in $2U\oplus k\latt{-1}$.  
If $k=4$, then there are three
such $\Tilde\Orth^+(2U\oplus D_4 )$-orbits, and  one  
of them coincides
with the  $-1$-vectors in $2U\oplus k\latt{-1}$.

The discriminant automorphic forms  of Table $3$  
(see \cite[\S 6]{GN9}) are modular with
respect to the full orthogonal group  $O^+(2U\oplus D_k )$ 
if $k\ne 4$
and with  a subgroup
$\Tilde O^+(2U\oplus D_4 )$ containing
$\Tilde O^+(U(2)\oplus (\latt{2}\oplus (5\latt{-2}))$.
If $k\le 5$, then the weight of the discriminant automorphic 
form is strictly  larger than
the dimension of the moduli space.

The similar argument works for the remaining cases with the modular 
forms constructed in \cite[\S 6.3--6.5]{GN9}.
\end{proof}

\noindent
{\bf Remark.} In each  Table $3$--$6$, there exists one 
discriminant automorphic form with
weight which is equal to the dimension  of the homogeneous domain. 
It follows that the Kodaira
dimension of a finite quotient of the corresponding  moduli space 
is equal to $0$. (See a criterion in \cite{G1} and \cite[Theorem 1.3]{GH}.)
We shall consider these cases in some details later.

Valery Gritsenko \par
\smallskip

Laboratoire Paul Painlev\'e et IUF\par
Universit\'e de Lille 1, France \par
%F59655 Villeneuve d'Ascq, Cedex, France \par
\smallskip

\par
National Research University Higher School of Economics
\par
Russian Federation
\par
\smallskip
Valery.Gritsenko@math.univ-lille1.fr
\vskip0.5cm

\vskip5pt

Viacheslav V. Nikulin

\par Steklov Mathematical Institute,\par
ul. Gubkina 8, Moscow 117966, GSP-1,
Russia \par
Deptm. of Pure Mathem. The University of
Liverpool, \par
Liverpool\ \ L69\ 3BX,\ UK  \par
nikulin@mi.ras.ru\ \ vvnikulin@list.ru\ \ vnikulin@liv.ac.uk
%%%%%%%%%%%%%%%%

\end{document}